\documentclass[12pt, reqno]{amsart}
\usepackage{amsmath, amsthm, amscd, amsfonts, amssymb, graphicx, color, mathrsfs}
\usepackage[bookmarksnumbered, colorlinks, plainpages]{hyperref}
\hypersetup{colorlinks=true,linkcolor=red, anchorcolor=green, citecolor=cyan, urlcolor=red, filecolor=magenta, pdftoolbar=true}
\usepackage{hyperref}
\newtheorem{theorem}{Theorem}[section]
\newtheorem{lemma}[theorem]{Lemma}

\newtheorem{corollary}[theorem]{Corollary}
\theoremstyle{definition}

\theoremstyle{remark}
\newtheorem{remark}[theorem]{Remark}
\numberwithin{equation}{section}

\begin{document}

\setcounter{page}{1}

\title[Inverse problem for a multi-term time-fractional diffusion equation \dots ]{Inverse problem for a multi-term time-fractional diffusion equation with the
Caputo derivatives
}

\author[R. Ashurov and  D. Shamuratov]{Ravshan Ashurov and Damir Shamuratov}

\address{\textcolor[rgb]{0.00,0.00,0.84}{Ravshan Ashurov:
V. I. Romanovskiy Institute of Mathematics Uzbekistan Academy of Sciences,
University street 9, Tashkent--100174,
Uzbekistan
}}
\email{\textcolor[rgb]{0.00,0.00,0.84}{ashurovr@gmail.com}}

\address{\textcolor[rgb]{0.00,0.00,0.84}{Damir Shamuratov:
V. I. Romanovskiy Institute of Mathematics Uzbekistan Academy of Sciences,
University street 9, Tashkent--100174,
Uzbekistan}}
\email{\textcolor[rgb]{0.00,0.00,0.84}{damirshamuratov5@gmail.com}}

\thanks{All authors contributed equally to the manuscript and read and approved the final manuscript.}



\begin{abstract} This paper investigates an inverse source problem for a multi-term time-fractional diffusion equation with Caputo derivatives. The source term is separable as \(f(x)g(t)\), with the unknown spatial component \(f(x)\) reconstructed from an overdetermination condition at interior time \(t_0 \in (0, T]\). The elliptic part is governed by a self-adjoint positive differential operator \(A(x, D)\) of order \(m \ge 2\). The solution features a spectral representation using the multinomial Mittag-Leffler function, for which we derive novel precise asymptotic expansions. These asymptotics provide a uniform lower bound for the solution's characteristic denominator, enabling sufficient conditions for the existence of a classical solution. Uniqueness of the reconstructed source holds under natural assumptions on the data and \(g(t)\). Despite the problem's ill-posedness, high-regularity classical solutions are achievable under suitable structural conditions.
\end{abstract} \maketitle
\tableofcontents

\section{Introduction}

Over the past few decades, fractional calculus has emerged as a powerful tool for modeling complex phenomena in natural and engineered systems. It has found applications across mathematics, physics, biology, finance, chemistry, and beyond \cite{1}, particularly for processes involving nonlocal interactions and memory effects \cite{2}–\cite{6}.

To capture even richer dynamics, multi-term time-fractional differential equations have been introduced. These incorporate multiple Caputo fractional derivatives of different orders, offering a more flexible framework than single-term models. For example, two-term fractional diffusion equations have been used to distinguish mobile and immobile phases in solute transport \cite{7},\cite{8}, while models with differing orders describe subdiffusive motion in velocity fields \cite{9},\cite{10}. Such formulations are especially effective for anomalous diffusion in heterogeneous media and complex systems, and they also facilitate numerical treatment of distributed-order derivatives.

Multi-term time-fractional diffusion equations have attracted significant attention in recent years. Numerical studies and approximations of strong solutions appear in \cite{11}, where the authors consider

$$
\sum_{j=1}^{r} \partial_t^{\rho_j} u(x,t) = -(-\Delta)^{\alpha/2} u(x,t) + q(x) f(t), \quad x \in \Omega, \ t \in (0,T],
$$
with the fractional Laplacian defined via spectral decomposition of \(-\Delta\) for \(1 < \alpha \le 2\). In the inverse setting, they recover the time-dependent source \(f(t)\) from interior measurement data \(u(x_0,t) = \phi(t)\). Note that all fractional derivatives in \cite{11} and throughout the present paper are understood in the Caputo sense (the precise definition is given below).

Galerkin-based numerical solutions for similar multi-term equations are developed in \cite{12}. A strong maximum principle for multi-term Caputo derivatives, combined with fractional Duhamel's principle, is established in \cite{13} and applied to recover time-dependent sources in problems of the form

$$
\sum_{j=1}^{m} a_j \partial_t^{\rho_j} u(x,t) + A u(x,t) = \rho(t) g(x), \quad x \in \Omega, \ 0 < t \le T,
$$
with homogeneous initial and Dirichlet boundary conditions, where \(A\) is a uniformly elliptic operator of second order.

Further works address parameter estimation \cite{14}, recovery of time-dependent coefficients from integral observations \cite{15}, Cauchy data \cite{16}, nonlocal overdetermination \cite{17}, and nonhomogeneous boundary conditions with partial boundary measurements \cite{18} using regularization techniques such as Levenberg-Marquardt.

Recall that the fractional derivatives $\partial_t^{\rho_j}$ in the Caputo sense are defined as follows:
\[
\partial_t^{\rho_j} u(x,t)
=
\frac{1}{\Gamma(1-\rho_j)}
\int_0^t
\frac{\partial u(x,s)}{\partial s}
\frac{ds}{(t-s)^{\rho_j}},
\qquad 0<\rho_j<1,\; t>0,
\]
provided that the right-hand side exists. Here $\Gamma$ denotes the Gamma function.

Most existing studies on inverse source problems for multi-term time-fractional diffusion equations focus on recovering time-dependent source terms and rely on specific spatial operators (classical Laplacian or its fractional powers). In contrast, the present paper investigates the inverse problem of identifying the spatially dependent source component \(f(x)\) in a separable source \(f(x)g(t)\), using a single overdetermination condition at a fixed interior time \(t_0 \in (0,T]\).

We consider the classical solution of the following initial-boundary value inverse problem:
\begin{equation}\label{1.1}
 \sum_{j=1}^M q_j  \partial_t^{\rho_j} u(x,t) + A(x,D)u(x,t) = f(x)g(t), \quad x \in \Omega, \; 0 < t \le T.
\end{equation}
\begin{equation}\label{1.2}
    u(x,0) = \varphi(x), \quad x \in \Omega,
\end{equation}
\begin{equation}\label{1.3}
B_j u(x,t) = \sum_{|\alpha|\le m_j} b_{\alpha,j}(x) D^\alpha u(x,t) = 0,\end{equation}
 $$ 0 \le m_j \le m-1, \quad j=1,\dots,l, \quad x \in \partial \Omega, \; 0 < t \le T. $$

\begin{equation}\label{1.4}
    u(x, t_0) = \Psi(x), \quad x \in \Omega,
\end{equation}
where $t_0$ is a given fixed point of the segment $(0, T]$.  \(A(x,D)\) is a positive, formally self-adjoint elliptic differential operator of even order \(m = 2l \ge 2\) with smooth coefficients, the orders satisfy \(0 < \rho_M < \cdots < \rho_1 < 1\), \(q_j > 0\), $j=1,\dots.M$ (with \(q_1 = 1\) without loss of generality), and all given functions are sufficiently smooth.

The main contribution is the construction of classical solutions via eigenfunction expansions, combined with a detailed asymptotic analysis of the multinomial Mittag-Leffler function--including precise expansions derived here for the first time. These asymptotics provide uniform lower bounds on the characteristic denominators, enabling rigorous proofs of existence and uniqueness for classical solutions of high regularity, despite the inherent ill-posedness of the inverse problem.

The paper is organized as follows. Preliminaries are given in Section \ref{sec:2}. Asymptotic behavior of the multinomial Mittag-Leffler function is derived in Section \ref{sec:3}. In Section \ref{sec:4}, we establish lower estimates for the term occurring in the denominator of the solution. The results are obtained by analyzing
the sign-changing or sign-preserving nature of the function $g(t)$, together with
the asymptotic behavior of the multinomial Mittag–Leffler function proved in
Section \ref{sec:3}. In Section \ref{sec:5}, the Fourier method is applied to construct a classical solution to the problem, and the existence of the solution is rigorously established. In Section \ref{sec:6}, the uniqueness of the solution to the inverse problem is investigated. It is shown that uniqueness may fail if certain Fourier coefficients vanish, and necessary and sufficient solvability conditions in the case of non-uniqueness are derived. In Section \ref{sec:7}, the problem is studied on the N-dimensional torus. Finally, we give a conclusion in Section \ref{sec:8}.

\section{Preliminaries}
\label{sec:2}
In solving the inverse problem \eqref{1.1}--\eqref{1.4}, we employ the Fourier method.

Applying the Fourier method to the inverse problem \eqref{1.1}--\eqref{1.4} leads to the spectral problem
\begin{equation}\label{2.1}
    A(x, D)v(x) = \lambda v(x), \quad x \in \Omega,
\end{equation}
\begin{equation}\label{2.2}
    B_j v(x) = 0, \quad j = 1, \dots, l, \quad x \in \partial \Omega.
\end{equation}

The paper \cite{19},\cite{20} establishes sufficient conditions on the boundary $\partial \Omega$ of the domain $\Omega$ 
and the coefficients of the operators $A$ and $B_j$ ensuring the compactness of the corresponding inverse operator 
or, equivalently, the existence of a complete $L_2(\Omega)$-orthonormal system of eigenfunctions $\{v_k(x)\}$ 
and a countable set of positive eigenvalues $\lambda_k$ $(0 < \lambda_1 \le \lambda_2 \le \cdots \le \lambda_k\le \cdots
\to\infty)$ of problem \eqref{2.1}--\eqref{2.2}.

The following conditions ensure that the spectral problem has a discrete spectrum of positive eigenvalues 
and a complete $L_2(\Omega)$ orthonormal system of eigenfunctions \cite{19}:

\begin{enumerate}
    \item \textbf{Strong ellipticity:} the principal symbol of $A$ satisfies
    \[
        \sum_{|\alpha|=m} a_\alpha(x) \xi^\alpha \ge c |\xi|^m, \quad \forall \xi \in \mathbb{R}^N, \ x \in \bar{\Omega}, \ c>0.
    \]
    \item \textbf{Self-adjointness:} the top-order coefficients are real and symmetric, i.e.,
    $a_\alpha(x) = \overline{a_\alpha(x)}$ for $|\alpha|=m$.

    \item \textbf{Positivity:} 
        $(Av,v)_{L^2(\Omega)} \ge c \|v\|_{H^l(\Omega)}^2, \quad c>0.$
    \item \textbf{Complementing boundary conditions (Lopatinskii--Shapiro condition):} 
    each boundary operator $B_j$ must be compatible with the ellipticity of $A$, 
    preventing nontrivial decaying solutions of the associated half-space problem.
    Typical examples include:
    \begin{itemize}
 \item \textbf{Boundary conditions:}
    \item $m=2$: Dirichlet ($v=0$), Neumann ($\partial_\nu v=0$), Robin ($a v + b \partial_\nu v=0$);
    
    \item $m=4$ (biharmonic): clamped ($v=0,\ \partial_\nu v=0$), simply supported ($v=0,\ \Delta v=0$).
    
    \item $\forall\, m=2l$: \quad
    $\partial_\nu^{k} v\big|_{\partial\Omega}=0, 
    \qquad k=0,\dots,l-1.$
\end{itemize}

\item \textbf{Smoothness:} 
$\partial \Omega \in C^{2l}$, 
$a_\alpha(x) \in C^{m}(\overline{\Omega})$, 
and $b_{\alpha,j}(x) \in C^{m_j}(\partial \Omega)$.
\end{enumerate}

 Let $\tau$ be an arbitrary real number. In the space $L_{2}(\Omega)$, we introduce the operator
$\widehat{A}^{\tau}$ acting by the rule 
\[
\widehat{A}^{\tau} h(x)
=
\sum_{k=1}^{\infty} \lambda_k^{\tau} h_k v_k(x),
\qquad
h_k = (h, v_k).
\]

Obviously, the operator $\widehat{A}^{\tau}$ with domain
\[
D(\widehat{A}^{\tau})
=
\left\{
h \in L_{2}(\Omega) :
\sum_{k=1}^{\infty} \lambda_k^{2\tau} |h_k|^{2} < \infty
\right\},
\]
is self-adjoint. If $A$ is the operator on $L_{2}(\Omega)$ acting by the formula
$
Ah(x) = A(x,D)h(x)
$
with domain
$
D(A)
=
\left\{
h \in C^{m}(\Omega) :
B_j h(x) = 0,\; j=1,\dots,l,\; x \in \partial\Omega
\right\},
$
then the operator $\widehat{A} \equiv \widehat{A}^{1}$ is a self-adjoint extension of $A$ in
$L_{2}(\Omega)$ .

We introduce the following lemma, which guarantees the continuity of the operator 
$D^{\alpha}\widehat{A}^{-\sigma}$ from $L^{2}(\Omega)$ into $C(\Omega)$. 
This result allows us to estimate the corresponding Fourier series in the 
$C(\Omega)$-norm and justifies the term-by-term application of the operator 
$D^{\alpha}$.

\begin{lemma}\label{lem1} \cite{20}
Let
$
\sigma > \frac{|\alpha|}{m} + \frac{N}{2m}.
$
Then the operator $D^{\alpha}\widehat{A}^{-\sigma}$ acts (completely) continuously
from $L_{2}(\Omega)$ into $C(\Omega)$, and the estimate
\[
\| D^{\alpha}\widehat{A}^{-\sigma} h \|_{C(\Omega)}
\le
C \| h \|_{L^{2}(\Omega)},
\]
holds.
\end{lemma}

The solution of problem \eqref{1.1}–\eqref{1.4} involves the multivariable Mittag--Leffler function. Therefore, we first present the definition of the multivariable Mittag--Leffler function.

The multinomial Mittag--Leffler function is defined as \cite{21}
$$E_{(\beta_1,\dots,\beta_M),\beta_0}(z_1,\dots,z_M)
=
\sum_{k=0}^\infty
\sum_{k_1+\dots+k_M=k}
\binom{k}{k_1,\dots,k_M}
\frac{\prod_{j=1}^M z_j^{k_j}}{\Gamma\ \left(\beta_0+\sum_{j=1}^M \beta_j k_j\right)},$$
where we assume \(\beta_0 > 0\),\, \(0 < \beta_j < 1\), \(z_j \in \mathbb{C}\) (\(j = 1, \dots, M\)), and \(\binom{k}{k_1,\dots,k_M}\) denotes
the multinomial coefficient
\[
\binom{k}{k_1,\dots,k_M} := \frac{k!}{k_1! \cdots k_M!},\quad
\text{with} \quad
k = \sum_{j=1}^{M} k_j,
\]
where \(k_j\), \(1 \le j \le M\), are non--negative integers.

To derive an upper bound for the multinomial Mittag–Leffler function and to prove the convergence of the solution, we recall the following lemmas.
\begin{lemma} \label{lem2}\cite{22} Let $0<\beta<2$ and
$1>\rho_1>\cdots>\rho_M>0$ be given. Assume that $\frac{\rho_1\pi}{2}<\mu<\rho_1\pi,\,\,
\mu \le |\arg(z_1)| \le \pi,
$
and that there exists a constant $K>0$ such that
$
-K \le z_j < 0, j=2,\dots,M.
$
Then there exists a constant $C>0$, depending only on $\mu$, $K$, $\rho_j$ $(j=1,\dots,M)$
and $\beta$, such that
\[
\left|
E_{(\rho_1,\rho_1-\rho_2,\dots,\rho_1-\rho_M),\beta}
(z_1,\dots,z_M)
\right|
\le
\frac{C}{1+|z_1|}.
\]
\end{lemma}
\begin{lemma}\label{lem3}\cite{22}
Let 
$
1>\rho_1>\cdots>\rho_M>0.
$
\,Then, for the multi-index Mittag--Leffler function,
\[
\frac{d}{dt} \Bigl[ t^{\rho_1} E_{\rho',\rho_1+1}(-\lambda_k t^{\rho_1},*) \Bigr]
= t^{\rho_1-1} E_{\rho',\rho_1}(-\lambda_k t^{\rho_1},*),
\]
where, 
$$
E_{\rho',\rho_1}(-\lambda_k t^{\rho_1},*)=E_{(\rho_1,\rho_1-\rho_2,...,\rho_1-\rho_M),\rho_1}(-\lambda_k t^{\rho_1},-q_2t^{\rho_1-\rho_2},...,-q_Mt^{\rho_1-\rho_M}),
$$
$$
\rho'=(\rho_1,\rho_1-\rho_2,...,\rho_1-\rho_M).
$$
\end{lemma}
\begin{lemma}\label{lem4}\cite{23} Let 
$
1>\rho_1>\cdots>\rho_M>0.
$ Then the function $t^{\rho_1-1} E_{\rho^{'},\rho_1}(-\lambda_k t^{\rho_1},*)$ is positive
for $t>0$.\\
\end{lemma}

\section{Asymptotic behavior of the multinomial Mittag-Leffler function}
\label{sec:3}
 In this section, we establish and justify the asymptotic behavior of the following multinomial Mittag--Leffler function:
\begin{equation}\label{3.1}
E_{\rho',\beta}(z_1,\ldots,z_M)
=
\sum_{k=0}^\infty
\sum_{k_1+\dots+k_M=k}
\binom{k}{k_1,\dots,k_M}
\frac{\prod_{j=1}^M z_j^{k_j}}{\Gamma\left(\beta+\rho_1k_1+\sum_{j=2}^M (\rho_1-\rho_j) k_j\right)}.
\end{equation}
\[
\rho' = (\rho_1, \rho_1 - \rho_2, \ldots, \rho_1 - \rho_M).
\]
Asymptotic estimates for similar multi-index and multivariable Mittag-Leffler functions have been studied by several researchers. Examples include the following authors:

 In paper \cite{24} Virginia S. Kiryakova proved an asymptotic estimate for the multi-index Mittag--Leffler function. 
The function is defined by
\begin{equation}\label{3.2}
E_{(1/\rho_i);(\mu_i)}(z)
=\sum_{k=0}^{\infty}
\frac{z^k}
{\Gamma\!\left(\mu_1 + \frac{k}{\rho_1}\right)
 \cdots
 \Gamma\!\left(\mu_m + \frac{k}{\rho_m}\right)}
\end{equation}
where $m \ge 1$ is an integer, $\rho_1, \ldots, \rho_m > 0$, and 
$\mu_1, \ldots, \mu_m$ are arbitrary real numbers.
The following asymptotic estimate holds for the multi-index Mittag--Leffler function \eqref{3.2}:
\[
\left| E_{(1/\rho_i);(\mu_i)}(z) \right|
\le
\exp\!\left( (\sigma + \varepsilon)\, |z|^{\rho} \right),
\qquad |z| \ge r_0(\varepsilon),
\]
where
\[
\frac{1}{\rho}
=
\frac{1}{\rho_1}
+\cdots+
\frac{1}{\rho_m},
\]
and
\[
\sigma
=
\left( \frac{\rho_1}{\rho} \right)^{\rho/\rho_1}
\cdots
\left( \frac{\rho_m}{\rho} \right)^{\rho/\rho_m}.
\]
Here $r_0(\varepsilon) > 0$ is sufficiently large, that is, 
the estimate holds for all $|z| \ge r_0(\varepsilon)$.

Furthermore, in \cite{25}, Christian Lavault established the asymptotic behaviour of the two-variable Mittag-Leffler function
$$
E_{\alpha,\beta}(x, y; \mu) = 
\sum_{n,m=0}^{\infty} \frac{x^n y^m}{\Gamma(n\alpha + m\beta + \mu)}, 
\quad (\alpha, \beta \in \mathbb{R}, \ \alpha, \beta > 0, \ \mu \in \mathbb{C}).
$$

In this paper, we establish the asymptotic behavior of the Mittag-Leffler multinomial
function in form \eqref{3.1}. We note that this behavior allows us to derive a uniform
lower bound for the spectral denominator appearing in the representation of the solution to the inverse problem. To our knowledge, this is the first time such an asymptotic estimate has been established. It should also be emphasized that this asymptotic estimate
is undoubtedly of independent interest.

The following theorem presents the asymptotic behavior for the multinomial Mittag–Leffler function in the form of \eqref{3.1}.

\begin{theorem}\label{Mitta-Leffler}
 Let $ \beta > 2\rho_1$ and $1 > \rho_1 > \cdots > \rho_M > 0$ be given. 
Assume that $\rho_1 \pi/2 < \mu < \rho_1 \pi$, 
$\mu \le |\arg(z_1)| \le \pi$, $|z_1|\to\infty$
and there exists $K > 0$ such that $-K \le z_j < 0 \; (j = 2, \ldots, M)$. Then we have the following asymptotic formulas in which $p$ is an arbitrary positive integer:

$$
E_{{\rho}',\beta}(z_1, \ldots, z_M)=-\frac{1}{z_1}\frac{1}{\Gamma(\beta-\rho_1)}-\frac{1}{z_1^2}\left(\frac{1}{\Gamma(\beta-2\rho_1)}-
\sum_{j=2}^{M}
\frac{z_j}{\Gamma(\beta-\rho_1-\rho_j)}\right)$$
$$
-\sum_{k=3}^p \frac{C_k(z_2,..,z_M,\rho_j,\beta)}{z_1^k}+O\left(|z_1|^{-p-1}\right).
$$
\end{theorem}
\begin{proof}
We have the following integral representation \cite{22}:
\begin{equation}\label{integral representation}
E_{{\rho}',\beta}(z_1,\ldots,z_M)
=
\frac{1}{2\rho_1 \pi i}
\int_{\gamma(R,\theta)}
\frac{
\exp(s^{1/\rho_1})
s^{(1-\beta)/\rho_1}
}{
s
-
z_1
-
\sum_{j=2}^{M}
z_j s^{\rho_j/\rho_1}
}
\, ds.
\end{equation}
where $\rho_1 \pi/2 < \theta < \mu$
and $\gamma(R,\theta)$ denotes the contour
\[
\gamma(R,\theta)
:=
\left\{
s \in {C}
\; ; \;
|s| = R,\; |\arg(s)| \le \theta
\right\}
\;\cup\;
\left\{
s \in {C}
\; ; \;
|s| > R,\; |\arg(s)| = \pm \theta
\right\},\] and , 
\[
R > K + K \sum_{j=2}^{M} R^{\rho_j / \rho_1}.
\]
In \cite{22}, $\beta$ is restricted to the interval $(0,2)$; however, as shown in \cite{26}, the integral representation \eqref{integral representation} holds for all $\beta>0$.

Now we introduce the following notation:
$$
Q(z_2,..,z_M,s)=s  - \displaystyle\sum_{j=2}^{M} z_j s^{\rho_j/\rho_1}, 
$$
then, we have
$$
E_{{\rho}', \,\beta}(z_1, \ldots, z_M)
=
\frac{1}{2\rho_1 \pi i}
\int_{\gamma(R,\theta)}
\frac{
\exp\!\left(s^{1/\rho_1}\right)
\, s^{\frac{1-\beta}{\rho_1}}
}{
Q(z_2,..,z_M,s)-z_1
}
\, ds.
$$
For any \( a \neq 1 \), the following identity holds:
\[
\frac{1}{1-a}
=
\sum_{k=1}^{p} a^{k-1}
+
\frac{a^p}{1-a}.
\]
To prove this, we observe that
\[
(1-a)\left(1 + a+ \cdots + a^{p-1}\right)
=
1 - a^p.
\]
Dividing both sides by \(1-a\) (for \(a \neq 1\)) yields the result.

 Let us write the expression in the following form:
$$
\frac{1}{Q(z_2,..,z_M,s)-z_1}=-\frac{1}{z_1}\frac{1}{1-\frac{Q(z_2,..,z_M,s)}{z_1}}$$
$$=-\sum_{k=1}^{p}\frac{Q^{k-1}(z_2,..,z_M,s)}{z_1^k}+\frac{Q^p(z_2,..,z_M,s)}{z_1^p(Q(z_2,..,z_M,s)-z_1)}
$$
$$
=-\frac{1}{z_1}-\frac{s-\displaystyle\sum_{j=2}^{M} z_j s^{\rho_j/\rho_1}}{z_1^2}-\sum_{k=3}^{p}\frac{Q^{k-1}(z_2,..,z_M,s)}{z_1^k}+\frac{Q^p (z_2,..,z_M,s)}{z_1^p(Q(z_2,..,z_M,s)-z_1)}.
$$
We substitute this into the integral,
$$
E_{{\rho}',\beta}(z_1,\ldots,z_M)
=
\frac{1}{2\rho_1 \pi i}
\int_{\gamma(R,\theta)}
\exp\!\left(s^{1/\rho_1}\right)
\, s^{\frac{1-\beta}{\rho_1}}\times
$$
$$\quad \times
\left(
-\frac{1}{z_1}
-\frac{s-\displaystyle\sum_{j=2}^{M} z_j s^{\rho_j/\rho_1}}{z_1^2}
-\sum_{k=3}^{p}\frac{Q^{k-1}(z_2,..,z_M,s)}{z_1^k}
+\frac{Q^p(z_2,..,z_M,s)}{z_1^p(Q(z_2,..,z_M,s)-z_1)}
\right)ds,
$$
or
$$
E_{{\rho}',\beta}(z_1,\ldots,z_M)=-\frac{1}{z_1}\frac{1}{2\rho_1 \pi i}
\int_{\gamma(R,\theta)}
\exp\!\left(s^{1/\rho_1}\right)
\, s^{\frac{1-\beta}{\rho_1}}ds
$$
$$
-\frac{1}{z^2_1}\frac{1}{2\rho_1 \pi i}
\int_{\gamma(R,\theta)}
\exp\!\left(s^{1/\rho_1}\right)
\, s^{\frac{1-\beta}{\rho_1}}\left(s-\displaystyle\sum_{j=2}^{M} z_j s^{\rho_j/\rho_1}\right)ds
$$
$$
-\frac{1}{2\rho_1 \pi i}
\int_{\gamma(R,\theta)}
\exp\!\left(s^{1/\rho_1}\right)
\, s^{\frac{1-\beta}{\rho_1}}\left(\sum_{k=3}^{p}\frac{Q^{k-1}(z_2,..,z_M,s)}{z_1^k}\right)ds
$$
$$
+\frac{1}{2\rho_1 \pi i}
\int_{\gamma(R,\theta)}
\exp\!\left(s^{1/\rho_1}\right)
\, s^{\frac{1-\beta}{\rho_1}}\frac{Q^p (z_2,..,z_M,s)}{z_1^p(Q(z_2,..,z_M,s)-z_1)}ds.
$$
Using the integral representation of $1/\Gamma(z)$ (see\,\cite{3}, p. 16) , we derive the following expressions involving the Gamma function for $\beta>2\rho_1$,
$$
E_{{\rho}',\beta}(z_1,\ldots,z_M)=-\frac{1}{z_1}\frac{1}{\Gamma(\beta-\rho_1)}$$
$$-\frac{1}{z_1^2}\left(\frac{1}{\Gamma(\beta-2\rho_1)}
-
\sum_{j=2}^{M}
\frac{z_j}{\Gamma(\beta-\rho_1-\rho_j)}\right)-\sum_{k=3}^p \frac{C_k(z_2,..,z_M,\rho_j,\beta)}{z_1^k}
$$
$$
+\frac{1}{2\rho_1 \pi i}
\int_{\gamma(R,\theta)}
\exp\!\left(s^{1/\rho_1}\right)
\, s^{\frac{1-\beta}{\rho_1}}\frac{Q^p (z_2,..,z_M,s)}{z_1^p(Q(z_2,..,z_M,s)-z_1)}ds.
$$
where
$$
C_k(z_2,..,z_M,\rho_j,\beta)=\frac{1}{2\rho_1 \pi i}
\int_{\gamma(R,\theta)}
\exp\!\left(s^{1/\rho_1}\right)
\, s^{\frac{1-\beta}{\rho_1}}Q^{k-1}(z_2,..,z_M,s)ds.
$$

We introduce the following notation
$$
I_p=\frac{1}{2\rho_1 \pi i}
\int_{\gamma(R,\theta)}
\exp\!\left(s^{1/\rho_1}\right)
\, s^{\frac{1-\beta}{\rho_1}}\frac{Q^p (z_2,..,z_M,s)}{z_1^p(Q(z_2,..,z_M,s)-z_1)}ds.$$
Since \( z_j < 0 \) and \( \rho_j < \rho_1 \) for \( j = 2, \ldots, M \), it turns out that the curve 
$
s - \sum_{j=2}^M z_j s^{\frac{\rho_j}{\rho_1}} \quad (s \in \gamma(R, \theta))
$
locates on the right-hand side of \( \gamma(R, \theta) \),
that is, \( \gamma(R, \theta) \) is shifted by the term 
$
- \sum_{j=2}^M z_j s^{\frac{\rho_j}{\rho_1}} 
$ to the positive direction.
This observation immediately implies,
\[
\min_{s \in \gamma(R, \theta)} \left| s - z_1 - \sum_{j=2}^M z_j s^{\frac{\rho_j}{\rho_1}} \right| \geq \min_{s \in \gamma(R, \theta)} |s- z_1| \geq |z_1| \sin(\mu - \theta).
\]
Therefore, we obtain the estimate
\[
|I_p| \leq \frac{|z_1|^{-p-1}}{2\rho_1\pi \sin\left(\mu - \theta\right)}  \int_{\gamma(R,\theta)} \left| \exp\left(s^{1/\rho_1}\right) \right| \left| s^{(1-\beta)/\rho_1} \right| \left|Q^p (z_2,..,z_M,s)\right|ds.
\]
The integral along $\gamma(R, \theta)$ converges, because for $s$ such that $\arg(s) = \pm \theta$ and $|s| > R$, there holds
\[
\left| \exp\left( s^{1/\rho_1} \right) \right| = \exp\left( |s|^{1/\rho_1} \cos\left( \frac{\theta}{\rho_1} \right) \right),
\]
with $\cos\left( \frac{\theta}{\rho_1} \right) < 0$.
Consider the circular part of $\gamma(R,\theta)$ defined by
\[
\{ s\in{C} : |s|=R,\ |\arg s|\le\theta \}.
\]
On this set we have $|s|=R$, and therefore the functions
$e^{s^{1/\rho_1}}$, $s^{(1-\beta)/\rho_1}$, and $Q(z_2,..,z_M,s)$  are bounded.
Hence there exists a constant $C_1>0$ such that
\[
\big| e^{s^{1/\rho_1}} s^{\frac{1-\beta}{\rho_1}} Q^p(z_2,..,z_M,s) \big|
\le C_1.
\]
Moreover, we have
\[
|Q(z_2,..,z_M,s)-z_1|\ge |z_1| \sin(\mu - \theta).
\]
Therefore,
\[
\left|
\frac{1}{2\rho_1\pi i}
\int_{|s|=R,\ |\arg s|\le\theta}
e^{s^{1/\rho_1}}
s^{\frac{1-\beta}{\rho_1}}
\frac{Q^p(z_2,..,z_M,s)}{z_1^p (Q(z_2,..,z_M,s)-z_1)}
\,ds
\right|
\le
\frac{C}{|z_1|^{p+1}},
\]
where $C>0$ is independent of $z_1$.

 Then we have the following asymptotics,
$$
E_{{\rho}',\beta}(z_1, \ldots, z_M)=-\frac{1}{z_1}\frac{1}{\Gamma(\beta-\rho_1)}-\frac{1}{z_1^2}\left(\frac{1}{\Gamma(\beta-2\rho_1)}
-
\sum_{j=2}^{M}
\frac{z_j}{\Gamma(\beta-\rho_1-\rho_j)}\right)$$
$$
-\sum_{k=3}^p \frac{C_k(z_2,..,z_M,\rho_j,\beta)}{z_1^k}+O\left(|z_1|^{-p-1}\right).
$$
This completes the proof. \end{proof}

\section{Lower bound for an infinitesimal denominator}
\label{sec:4}

In the solution of the inverse problem \eqref{1.1}--\eqref{1.4}, the denominator contains the term
\[
b_{k,\rho_1}(t_0)
=
\int_0^{t_0}
g(t_0 - \xi)\,
\xi^{\rho_1-1}
E_{\rho',\rho_1}(-\lambda_k \xi^{\rho_1},\ast)
\, d\xi.
\]
In this section, we derive a lower bound for this term. The estimate is obtained by taking into account whether the function $g(t)$ preserves its sign or changes sign, together with the asymptotic behavior of the above multinomial Mittag--Leffler function.

\medskip

Let us divide the set of natural numbers $\mathbb{N}$ into two groups $K_{0,\rho_1}$ and $K_{\rho_1}$:
$\mathbb{N} = K_{\rho_1} \cup K_{0,\rho_1},$
where a number $k$ is assigned to $K_{0,\rho_1}$ if $b_{k,\rho_1}(t_0) = 0$, and if $b_{k,\rho_1}(t_0) \neq 0$, then the number is assigned to $K_{\rho_1}$.  
Note that for some $t_0$ the set $K_{0,\rho_1}$ can be empty, in which case $K_{\rho_1} = \mathbb{N}$. For example, if $g(t)$ is sign-preserving, then $ K_{\rho_1} = \mathbb{N}
$ for all $ t_0.$
\medskip
\begin{lemma}\label{lem5}
Let $\rho_j \in (0,1)$ for $j=1,\dots,M$, $g(t) \in C[0,T]$ and $g(t)\neq 0$ for $t\in[0,T]$ (that is, $g(t)$ does not change sign). Then the following estimates hold for all $k$ and $t_0$:
\begin{equation}\label{4.1}
 \frac{\widetilde C_0}{\lambda_k} \le |b_{k,\rho_1}(t_0)| \le  \frac{\widetilde C_1}{\lambda_k},
\end{equation}
where the constants $\widetilde C_0$ and $\widetilde C_1 > 0$.
\end{lemma}
\begin{proof}
    By the Weierstrass theorem, there exists a constant $g_0>0$ such that
$|g(t)| \ge g_0$.
We apply the mean value theorem and Lemma \ref{lem3} to obtain
\[
|b_{k,\rho_1}(t_0)|
=
\left|
\int_{0}^{t_0}
\xi^{\rho_1-1}
E_{\rho',\rho_1}(-\lambda_k \xi^{\rho_1},*)
\, g(t_0-\xi)\, d\xi
\right|
\]
\[=
|g(\eta_k)|\, t_0^{\rho_1}
E_{\rho',\rho_1+1}(-\lambda_k t_0^{\rho_1},*),
\quad
\eta_k \in [0,t_0].
\]
Applying the asymptotic estimate from Theorem \ref{Mitta-Leffler}, we obtain
\[
E_{\rho',\rho_1+1}(-\lambda_k t_0^{\rho_1},*)
=
\frac{1}{\lambda_k t_0^{\rho_1}}
+
O\!\left(\frac{1}{(\lambda_k t_0^{\rho_1})^2}\right).
\]
Hence, for sufficiently large $k\ge k_0$, using the estimate
$
|g(t)| \ge g_0,
$
one has
\[
|b_{k,\rho_1}(t_0)|
=
|g(\eta_k)|t_0^{\rho_1}
\left(\frac{1}{\lambda_kt_0^{\rho_1}}+O\left(\frac{1}{\left(\lambda_kt_0^{\rho_1}\right)^2}\right)\right)
\ge
\frac{C_0}{\lambda_k}.
\]

It remains to consider the finitely many indices $k<k_0$.
Since the function $g(t)$ does not change sign, it follows that $k \in K_{\rho_1}$. Consequently, we have
$
b_{k,\rho_1}(t_0) \neq 0
$ for all k.

For the finite set $\{k: k<k_0\}$,
the quantities
\[
\lambda_k |b_{k,\rho_1}(t_0)|
\]
form a finite collection of positive numbers.
Therefore their minimum exists and is strictly positive:
\[
\overline{C}=
\min_{k<k_0}
\lambda_k |b_{k,\rho_1}(t_0)|
>0.
\]
Consequently,
\[
|b_{k,\rho_1}(t_0)|
\ge
\frac{\overline{C}}{\lambda_k},
\qquad
k<k_0.
\]
We define
$
\widetilde C_0=\min\{C_0,\overline{C}\}
$. Then 
\[
|b_{k,\rho_1}(t_0)|
\ge
\frac{\widetilde C_0}{\lambda_k}
\quad
\text{for all } k.
\]
The upper estimate follows from Lemma \ref{lem2}:
\[
|b_{k,\rho_1}(t_0)|
\le
C\, |g(\eta_k)|
\frac{t_0^{\rho}}{1 + \lambda_k t_0^{\rho}}
\le
C \, \max_{0 \le \eta_k \le t_0} |g(\eta_k)|
\frac{1}{\lambda_k}
\le
\frac{\widetilde C_1}{\lambda_k}.
\]
This completes the proof.
\end{proof}

\begin{lemma}\label{lem6}  Let $\rho_j \in (0,1)$ for $j=1,\dots,M$, $g(t) \in C^1[0,T]$ and $g(0)\neq 0$ ($g(t)$  changes sign). Then there exist numbers $r_0 > 0$ and $k_0$ such that, for all $t_0 \le r_0$ and $k \ge k_0$, the following estimates hold:
\begin{equation}\label{4.2}
 \frac{\widetilde C_0}{\lambda_k} \le |b_{k,\rho_1}(t_0)| \le  \frac{\widetilde C_1}{\lambda_k},
\end{equation}
where the constants $\widetilde C_0$ and $\widetilde C_1 > 0$ depend on $r_0$ and $k_0$.\end{lemma}
\begin{proof} 
Let $\rho_j \in (0,1)$. Using Lemma \ref{lem3} we integrate by parts and then apply the mean value theorem. Then we have
$$
b_{k,\rho_1}(t_0) 
= \int_0^{t_0} g(t_0 - \xi) \, \xi^{\rho_1-1} E_{\rho',\rho_1}(-\lambda_k \xi^{\rho_1},*) \, d\xi $$
$$= \int_0^{t_0} g(t_0 - \xi) \, d\big[ \xi^{\rho_1} E_{\rho',\rho_1+1}(-\lambda_k \xi^{\rho_1},*) \big] $$
$$= g(t_0 - \xi) \, \xi^{\rho_1} E_{\rho',\rho_1+1}(-\lambda_k \xi^{\rho_1},*) \Big|_{0}^{t_0} 
+ \int_0^{t_0} g'(t_0 - \xi) \, \xi^{\rho_1} E_{\rho',\rho_1+1}(-\lambda_k \xi^{\rho_1},*) \, d\xi $$
$$= g(0) t_0^{\rho_1} E_{\rho',\rho_1+1}(-\lambda_k t_0^{\rho_1},*) + g'(\eta_k) \int_0^{t_0} \xi^{\rho_1} E_{\rho',\rho_1+1}(-\lambda_k \xi^{\rho_1},*) \, d\xi, \quad \eta_k \in [0,t_0].
$$
For the last integral, using Lemma \ref{lem3}
\[
\int_0^{t_0} \xi^{\rho_1} E_{\rho',\rho_1+1}(-\lambda_k \xi^{\rho_1},*) \, d\xi = t_0^{\rho_1+1} E_{\rho',\rho_1+2}(-\lambda_k t_0^{\rho_1},*).
\]
Applying the Theorem \ref{Mitta-Leffler}, 
$$
E_{\rho',\rho_1+1}\!\left(-\lambda_k t_0^{\rho_1},*\right)
=
\frac{1}{\lambda_k t_0^{\rho_1}} +O\left(\frac{1}{(\lambda_k t^{\rho_1}_0)^{2}}\right),
 $$
$$
E_{\rho',\rho_1+2}\!\left(-\lambda_k t_0^{\rho_1},*\right)
=
\frac{1}{\lambda_k t_0^{\rho_1}} +O\left(\frac{1}{(\lambda_k t^{\rho_1}_0)^{2}}\right),
$$
we obtain
\[
b_{k,\rho_1}(t_0) = \frac{g(0)}{\lambda_k} + \frac{g'(\eta_k) t_0}{\lambda_k } +  O\left(\frac{1}{(\lambda_k t^{\rho_1}_0)^{2}}\right).
\]

If $g(0) \neq 0$, sufficiently small $t_0$  and sufficiently large $k$ we obtain the required lower estimate. This also implies the corresponding upper bound. 
\end{proof}
\begin{corollary}\label{cor1}
If the conditions of Lemma \ref{lem6} are satisfied, then the estimate \eqref{4.2} holds for all $t_0 \le r_0$ and $k \in K_{\rho_1}$.
\end{corollary}
  It is not hard to see that, this statement follow from Lemma \ref{lem5}.

\begin{corollary}\label{cor2}
If the conditions of Lemma \ref{lem6} are satisfied and $t_0$ is sufficiently small, then the set $K_{0,\rho_1}$ has a finite number of elements.
\end{corollary}

\begin{proof}
By Lemma \ref{lem6}, for sufficiently small $t_0 \le r_0$ and sufficiently large $k \ge k_0$, we have
\[
b_{k,\rho_1}(t_0) = \frac{g(0)}{\lambda_k} + \frac{g'(\eta_k) t_0}{\lambda_k } +  O\left(\frac{1}{(\lambda_k t^{\rho_1}_0)^{2}}\right) \neq 0,
\]
since $g(0) \neq 0$. 

Therefore, the equality $b_{k,\rho_1}(t_0) = 0$ can only hold for finitely many indices $k < k_0$. 
This implies that the set
\[
K_{0,\rho_1} = \{ k \in \mathbb{N} : b_{k,\rho_1}(t_0) = 0 \},
\]
contains only finitely many elements. 
\end{proof}

\section{Existence}
\label{sec:5}
According to the Fourier method, we seek the function $u(x, t)$ as a formal series:
\begin{equation}\label{5.1}
u(x, t) = \sum_{k=1}^{\infty} T_k(t) \, v_k(x),
\end{equation}
here the functions $T_k(t)$ are a solution of the Cauchy type problem
\begin{equation}\label{5.2}
\sum_{j=1}^M q_j\partial_t^{\rho_{j}} T_k(t) + \lambda_k T_k(t) = f_kg(t),  \end{equation}
\begin{equation}\label{5.3}
T_k(0) = \varphi_k. 
\end{equation}
Note that the unknowns in problem \eqref{5.2} are also the coefficients $f_k$. To determine these numbers, from \eqref{1.4} we obtain the additional condition
\begin{equation}\label{5.4}
    T_k(t_0)=\Psi_k.
\end{equation}
The solution of problem \eqref{5.2}--\eqref{5.3} has the form \cite{11},
\begin{equation}\label{5.5}
T_k(t) = \varphi_k \left[ 1 - \lambda_k t^{\rho_1} E_{ \rho', \rho_1 + 1} \left( -\lambda_k t^{\rho_1}, * \right) \right]
+ f_kb_{k,\rho_1}(t).
\end{equation}
Then we obtain the equation
$$T_k(t_0) = \varphi_k \left[ 1 - \lambda_k {t_0}^{\rho_1} E_{ \rho', \rho_1 + 1} \left( -\lambda_k {t_0}^{\rho_1}, * \right) \right]+ f_kb_{k,\rho_1}(t_0)=\Psi_k.$$ 
Consequently,
\begin{equation}\label{5.6}
f_kb_{k,\rho_1}(t_0)=\Psi_k-\varphi_k \left[ 1 - \lambda_k t_0^{\rho_1} E_{ \rho', \rho_1 + 1} \left( -\lambda_k t_0^{\rho_1}, * \right) \right],
\end{equation}
\begin{equation}\label{5.7}
f_k=\frac{\Psi_k-\varphi_k \left[ 1 - \lambda_k t_0^{\rho_1} E_{ \rho', \rho_1 + 1} \left( -\lambda_k t_0^{\rho_1}, * \right) \right]}{b_{k,\rho_1}(t_0)}.
\end{equation}

Thus, we have constructed the formal series
\begin{equation}\label{5.8}
u(x,t)= \sum_{k=1}^{\infty}\varphi_k \left[ 1 - \lambda_k t^{\rho_1} E_{ \rho', \rho_1 + 1} \left( -\lambda_k t^{\rho_1}, * \right) \right]v_k(x)+\sum_{k=1}^{\infty}f_kb_{k,\rho_1}(t) v_k(x),
\end{equation}
 and
\begin{equation}\label{5.9}
f(x)=\sum_{k=1}^{\infty}\frac{\Psi_k v_k(x)}{b_{k,\rho_1}(t_0)}-\sum_{k=1}^{\infty}\frac{\varphi_k \left[ 1 - \lambda_k t_0^{\rho_1} E_{ \rho', \rho_1 + 1} \left( -\lambda_k t_0^{\rho_1}, * \right) \right]v_k(x)}{b_{k,\rho_1}(t_0)}.
\end{equation}
Obviously, the functions \eqref{5.8} and \eqref{5.9} are formal solutions. It only remains to justify the Fourier method.

\begin{theorem}\label{theorem2} Let $\varphi \in D(\widehat{A}^{\tau+1})$ and $\Psi \in D(\widehat{A}^{\tau+1})$, where
$
\tau > \frac{N}{2m}$, and $g(t) \in C[0,T]$, \,\, $g(t) \neq 0$.
Then there exists a classical solution $(u(x,t), f(x))$ of the inverse problem \eqref{1.1}--\eqref{1.4}, which can be represented in the form of the series \eqref{5.8} and \eqref{5.9}.
\end{theorem}
\begin{proof}
Let the assumptions of the theorem be satisfied. First, let us prove the absolute and uniform convergence of the series \eqref{5.9} for the function $f$. Set
$$
f_n^1(x)=\sum_{k=1}^{n}\frac{\Psi_k v_k(x)}{b_{k,\rho_1}(t_0)},
$$
$$
f_n^2(x)=\sum_{k=1}^{n}\frac{\varphi_k \left[ 1 - \lambda_k t_0^{\rho_1} E_{ \rho', \rho_1 + 1} \left( -\lambda_k t_0^{\rho_1}, * \right) \right]v_k(x)}{b_{k,\rho_1}(t_0)}.
$$
The condition of the theorem implies the convergence of the series
\[
\sum_{k=1}^{\infty} \lambda_k^{2(\tau+1)} |\Psi_k|^2 \le C_{\Psi} < \infty.
\]
For the function $\Psi$ for some $\tau > \frac{N}{2m}$, and for the convergence of the series
\[
\sum_{k=1}^{\infty} \lambda_k^{2(\tau+1)} |\varphi_k|^2 \le C_{\varphi}< \infty,
\]
for the function $\varphi$.

Then, applying Lemma \ref{lem1} with $\alpha=0$ and using the identity 
$\widehat{A}^{-\tau} v_k(x)=\lambda_k^{-\tau} v_k(x)$, we obtain
 \begin{equation}\label{5.10}
\| f_n^1(x) \|_{C(\Omega)}
=
\left\|
 \widehat{A}^{-\tau} \sum_{k=1}^{n}\frac{\lambda_k^{\tau}\Psi_k v_k(x)}{b_{k,\rho_1}(t_0)}
\right\|_{C(\Omega)}
\le C \,
\left\|
\sum_{k=1}^{n}\frac{\lambda_k^{\tau}\Psi_k v_k(x)}{b_{k,\rho_1}(t_0)}
\right\|_{L^2(\Omega)}.\end{equation}

 By the orthonormality of the system ${v_k(x)}$, we have
$$
\| f_n^1(x) \|_{C(\Omega)}^2
\le
C 
\sum_{k=1}^{n}\left|\frac{\lambda_k^{\tau}\Psi_k}{b_{k,\rho_1}(t_0)}
\right|^2.
$$
Further, applying the Lemma \ref{lem5} (lower bound), we obtain
\begin{equation} \label{5.11}
\| f_n^1(x) \|_{C(\Omega)}^2
\le
C\sum_{k=1}^{\infty} \lambda_k^{2(\tau+1)} |\Psi_k|^2 \le C\cdot C_{\Psi} < \infty.
\end{equation}
Using the same argument as above to the function $f_n^2(x)$ , we arrive at the inequality

 $$
\| f_n^2(x) \|_{C(\Omega)}
=
\left\|
 \widehat{A}^{-\tau} \sum_{k=1}^{n}\frac{\lambda_k^\tau\varphi_k \left[ 1 - \lambda_k t_0^{\rho_1} E_{\rho', \rho_1 + 1} \left( -\lambda_k t_0^{\rho_1}, * \right) \right]v_k(x)}{b_{k,\rho_1}(t_0)}
\right\|_{C(\Omega)}$$
\begin{equation}\label{5.12}
\le C \,
\left\|
\sum_{k=1}^{n}\frac{\lambda_k^\tau\varphi_k \left[ 1 - \lambda_k t_0^{\rho_1} E_{ \rho', \rho_1 + 1} \left( -\lambda_k t_0^{\rho_1}, * \right) \right]v_k(x)}{b_{k,\rho_1}(t_0)}
\right\|_{L^2(\Omega)}.\end{equation}
Since the system ${v_k(x)}$ is orthonormal, it follows that
$$
\| f_n^2(x) \|_{C(\Omega)}^2
\le
C 
\sum_{k=1}^{n}\left|\frac{\lambda_k^\tau\varphi_k \left[ 1 - \lambda_k t_0^{\rho_1} E_{ \rho', \rho_1 + 1} \left( -\lambda_k t_0^{\rho_1}, * \right) \right]}{b_{k,\rho_1}(t_0)}
\right|^2.
$$

By Lemmas \ref{lem2}, \ref{lem4} and \ref{lem5} (lower bound), we get
\begin{equation}\label{5.13}
\| f_n^2(x) \|_{C(\Omega)}^2
 \le
C\sum_{k=1}^{\infty} \lambda_k^{2(\tau+1)} |\varphi_k|^2 \le C\cdot C_{\varphi} < \infty.
\end{equation}
The estimates \eqref{5.11} and \eqref{5.13} imply the uniform convergence of the series \eqref{5.9} defining the function $f(x)$. On the other hand, the sums in \eqref{5.10} and \eqref{5.12} converge under any permutation of their
terms, because these terms are mutually orthogonal. This implies the absolute convergence of the
series \eqref{5.9}.

Using the definition of the coefficients $f_k$ , we study the following three sums:

\begin{equation}\label{5.14}
u_n^1(x,t)=\sum_{k=1}^{n}\varphi_k \left[ 1 - \lambda_k t^{\rho_1} E_{ \rho', \rho_1 + 1} \left( -\lambda_k t^{\rho_1}, * \right) \right]v_k(x),
\end{equation}
\begin{equation}\label{5.15}
u_n^2(x,t)=\sum_{k=1}^{n} \frac{\Psi_kb_{k,\rho_1}(t) v_k(x)}{b_{k,\rho_1}(t_0)},
\end{equation}
\begin{equation}\label{5.16}
u_n^3(x,t)=\sum_{k=1}^n \frac{\varphi_k \left[ 1 - \lambda_k t^{\rho_1} E_{ \rho', \rho_1 + 1} \left( -\lambda_k t^{\rho_1}, * \right) \right]b_{k,\rho_1}(t)v_k(x)}{b_{k,\rho_1}(t_0)}.
\end{equation}

 Let $|\alpha|\le m$. We need to show that each of these sums converges uniformly
and absolutely in the domain $\overline{\Omega} \times (0 , T]$ after the operators $D^\alpha$ and
$\partial_t^{\rho_j}$ are applied term by term to the sum.

By Lemma \ref{lem1}, one has the estimate
\[
\| D^\alpha u_n^1 \|_{C(\Omega)}
=
\left\|
D^\alpha \widehat{A}^{-\tau-1}
\sum_{k=1}^{n}\lambda_k^{\tau+1}\varphi_k \left[ 1 - \lambda_k t^{\rho_1} E_{ \rho', \rho_1 + 1} \left( -\lambda_k t^{\rho_1}, * \right) \right]v_k(x)
\right\|_{C(\Omega)}
\]
\begin{equation}\label{5.17}
\le
C
\left\|
\sum_{k=1}^{n}\lambda_k^{\tau+1}\varphi_k \left[ 1 - \lambda_k t^{\rho_1} E_{ \rho', \rho_1 + 1} \left( -\lambda_k t^{\rho_1}, * \right) \right]v_k(x)
\right\|_{L^2(\Omega)} .
\end{equation}
Using the orthonormality of the system $v_k(x)$, we obtain
\[
\| D^\alpha u_n^1 \|_{C(\Omega)}^2
\le
C \sum_{k=1}^{n}
\left|
\lambda_k^{\tau+1}\varphi_k \left[ 1 - \lambda_k t^{\rho_1} E_{ \rho', \rho_1 + 1} \left( -\lambda_k t^{\rho_1}, * \right) \right]
\right|^2.
\]
Apply Lemma \ref{lem2}, to get
\[
\| D^\alpha u_n^1 \|_{C(\Omega)}^2
\le C\sum_{k=1}^{\infty} \lambda_k^{2(\tau+1)} |\varphi_k|^2 \le C\cdot C_{\varphi} < \infty.
\]
Similarly, for $u_n^2(x,t)$ , we have  
\[
\| D^\alpha u_n^2 \|_{C(\Omega)}
=
\left\|
D^\alpha \widehat{A}^{-\tau-1}
\sum_{k=1}^{n}\lambda_k^{\tau+1}\frac{\Psi_k b_{k,\rho_1}(t) v_k(x)}{b_{k,\rho_1}(t_0)}
\right\|_{C(\Omega)}
\]
\begin{equation}\label{5.18}
\le
C
\left\|
\sum_{k=1}^{n}\lambda_k^{\tau+1}\frac{\Psi_k b_{k,\rho_1}(t) v_k(x)}{b_{k,\rho_1}(t_0)}
\right\|_{L^2(\Omega)} .
\end{equation}
By the orthonormality of the system ${v_k(x)}$, one has
\[
\| D^\alpha u_n^2 \|_{C(\Omega)}^2
\le
C \sum_{k=1}^{n}
\left|
\lambda_k^{\tau+1}\frac{\Psi_k b_{k,\rho_1}(t)}{b_{k,\rho_1}(t_0)}
\right|^2.
\]
Let $ g \in  C[0,T] $ and define
$\|g\| = \displaystyle \max_{0 \le t \le T} |g(t)|.$
Then,
\[
\left|
\int_0^t (t-\xi)^{\rho_1 - 1}
E_{\rho',\rho_1}\!\left(
-\lambda_k (t-\xi)^{\rho_1}
\right)
g(\xi)\, d\xi
\right|\]
\[\le
\int_0^t
\left|
(t-\xi)^{\rho_1 - 1}
E_{\rho',\rho_1}\!\left(
-\lambda_k (t-\xi)^{\rho_1},*
\right)
\right|
\, |g(\xi)|\, d\xi
\]
\[
\le
\|g\|
\int_0^t
\left|
(t-\xi)^{\rho_1 - 1}
E_{\rho',\rho_1}\!\left(
-\lambda_k (t-\xi)^{\rho_1},*
\right)
\right|
\, d\xi .
\]
Consequently,
\[
\left|
\int_0^t (t-\xi)^{\rho_1 - 1}
E_{\rho',\rho_1}\!\left(
-\lambda_k (t-\xi)^{\rho_1},*
\right)
g(\xi)\, d\xi
\right|^2
\]\[
\le \|g\|^2
\left(
\int_0^t
\left|
(t-\xi)^{\rho_1 - 1}
E_{\rho',\rho_1}\!\left(
-\lambda_k (t-\xi)^{\rho_1},*
\right)
\right|
\, d\xi
\right)^2.
\]
Using Lemmas \ref{lem2}, \ref{lem3}, and \ref{lem5} (lower bound), we obtain 
\[
\| D^\alpha u_n^2 \|_{C(\Omega)}^2
\le
C||g||^2 \sum_{k=1}^{\infty} \lambda_k^{2(\tau+1)} |\Psi_k|^2 \le C\cdot ||g||^2\cdot C_{\Psi} < \infty.\]

For $u_n^3(x,t)$, the following estimate holds:
\[
\| D^\alpha u_n^3 \|_{C(\Omega)}=
\left\|
D^\alpha \widehat{A}^{-\tau-1}
\sum_{k=1}^n \frac{\lambda_k^{\tau+1}\varphi_k \left[ 1 - \lambda_k t^{\rho_1} E_{ \rho', \rho_1 + 1} \left( -\lambda_k t^{\rho_1}, * \right) \right]b_{k,\rho_1}(t) v_k(x)}{b_{k,\rho_1}(t_0)}
\right\|_{C(\Omega)}
\]
\begin{equation}\label{5.19}
\le
C
\left\|
\sum_{k=1}^{n}\frac{\lambda_k^{\tau+1}\varphi_k \left[ 1 - \lambda_k t^{\rho_1} E_{ \rho', \rho_1 + 1} \left( -\lambda_k t^{\rho_1}, * \right) \right]b_{k,\rho_1}(t) v_k(x)}{b_{k,\rho_1}(t_0)}
\right\|_{L^2(\Omega)} .
\end{equation}
By the orthonormality of the system $v_k(x)$, we get
\[
\| D^\alpha u_n^3 \|_{C(\Omega)}^2
\le
C \sum_{k=1}^{n}
\left|
\frac{\lambda_k^{\tau+1}\varphi_k \left[ 1 - \lambda_k t^{\rho_1} E_{ \rho', \rho_1 + 1} \left( -\lambda_k t^{\rho_1}, * \right) \right]b_{k,\rho_1}(t)}{b_{k,\rho_1}(t_0)}
\right|^2.
\]
Apply Lemmas \ref{lem2}, \ref{lem3}, \ref{lem4}, and \ref{lem5} (lower bound), together with the assumption that $g(t)\in C[0,T]$, to get
\[
\| D^\alpha u_n^3 \|_{C(\Omega)}^2 \le
C||g||^2 \sum_{k=1}^{\infty} \lambda_k^{2(\tau+1)} |\varphi_k|^2 \le C\cdot ||g||^2\cdot C_{\varphi} < \infty.\]

This implies the uniform convergence of the series 
\eqref{5.14}--\eqref{5.16} after applying the spatial derivatives 
$D^\alpha$, $|\alpha|\le m$, in the closed cylinder 
$\overline{\Omega}\times (0,T]$.
On the other hand, the series \eqref{5.17}--\eqref{5.19} 
converge under any permutation of their terms, because these terms are mutually orthogonal. 
Therefore, the differentiated series 
\eqref{5.14}--\eqref{5.16} converge absolutely.

Now,  it remains to verify the validity of the term-by-term application of the operator
$\partial_t^{\rho_j}$ to the series \eqref{5.8} defining the solution $u(x,t)$.
It can be readily verified that 
\[
\sum_{k=1}^{n}\sum_{j=1}^M q_j \partial_t^{\rho_j}T_k(t) v_k(x)
=
- \sum_{k=1}^{n} \lambda_k T_k(t) v_k(x)
+g(t) \sum_{k=1}^{n} f_k v_k(x)=
\]
\[
=
- A(x,D)\,
\widehat{A}^{-\tau-1}
\sum_{k=1}^{n}
\lambda_k^{\tau+1} T_k(t) v_k(x)
+ g(t)\widehat{A}^{-\tau}
\sum_{k=1}^{n}
\lambda_k^{\tau} f_k v_k(x).
\]
The absolute and uniform convergence of the series on the right-hand side has already been proved above.

The fact that the functions \eqref{5.8} and \eqref{5.9} satisfy all conditions \eqref{1.1}--\eqref{1.4} of the inverse problem
is obvious.
\end{proof}

\begin{theorem}\label{theorem3}
    Let $\varphi \in D(\widehat{A}^{\tau+1})$ and $\Psi \in D(\widehat{A}^{\tau+1})$, where
$
\tau > \frac{N}{2m}$, and $g(t) \in C^1[0,T]$, \,\, $g(0) \neq 0$.\\
    1) If the set $K_{0,\rho_1}$ is empty i.e. $b_{k,\rho_1}(t_0)\neq 0$, for all $k$ then there exists solution of the inverse problem \eqref{1.1}--\eqref{1.4}:
    $$
    f(x)=\sum_{k=1}^{\infty}\frac{\Psi_k v_k(x)}{b_{k,\rho_1}(t_0)}-\sum_{k=1}^{\infty}\frac{\varphi_k \left[ 1 - \lambda_k t_0^{\rho_1} E_{ \rho', \rho_1 + 1} \left( -\lambda_k t_0^{\rho_1}, * \right) \right]v_k(x)}{b_{k,\rho_1}(t_0)},
    $$
    $$
   u(x,t)= \sum_{k=1}^{\infty}\varphi_k \left[ 1 - \lambda_k t^{\rho_1} E_{ \rho', \rho_1 + 1} \left( -\lambda_k t^{\rho_1}, * \right) \right]v_k(x)+\sum_{k=1}^{\infty}f_kb_{k,\rho_1}(t) v_k(x).
    $$
    2)If the set $K_{0,\rho_1}$ is not empty, for the existence of a solution to the inverse problem, it is necessary and sufficient that the following conditions
    \begin{equation}\label{5.20}
\Psi_k = \varphi_k \left[ 1 - \lambda_k t_0^{\rho_1} E_{ \rho', \rho_1 + 1} \left( -\lambda_k t_0^{\rho_1}, * \right) \right],
\quad k \in K_{0,\rho_1}.
\end{equation}
    be satisfied. In this case, the solution to problem \eqref{1.1}–\eqref{1.4} can be written in the form:
\[
f(x)=\sum_{k\in K_{\rho_1}}
\frac{\Psi_k -\varphi_k \left[ 1 - \lambda_k t_0^{\rho_1} E_{\rho', \rho_1 + 1} \left( -\lambda_k t_0^{\rho_1}, * \right) \right]}
{b_{k,\rho_1}(t_0)}\, v_k(x)
+\sum_{k\in K_{0,\rho_1}} f_k v_k(x),
\]
\[
u(x,t)
=
\sum_{k=1}^{\infty}
\left(\varphi_k
\left[
1-\lambda_k t^{\rho_1}
E_{\rho',\rho_1+1}\!\left(-\lambda_k t^{\rho_1},*\right)
\right]+ f_k b_{k,\rho_1}(t)\right) v_k(x),
\]
where $f_k$, $k \in K_{0,\rho_1}$, are arbitrary real numbers.
\end{theorem}
\begin{proof}
The proof of the first part of the theorem is completely analogous to the proof of Theorem
\ref{theorem2}. As regards the proof of the second part of the theorem, we note the following:

If $k \in K_{\rho_1}$, then \eqref{5.5} and \eqref{5.7} follow directly from \eqref{5.6}.

On the other hand, if $k \in K_{0,\rho_1}$, i.e., $b_{k,\rho_1}(t_0) = 0$, equation \eqref{5.6} admits a solution for $f_k$ if and only if condition \eqref{5.20} is satisfied. In this case, the corresponding values of $f_k$ can be chosen arbitrarily.

Moreover, as noted in Corollaries \eqref{cor1} and \eqref{cor2}, the set $K_{0,\rho_1}$ contains only finitely many elements.
\end{proof}

\begin{remark}
    For conditions \eqref{5.20} to be satisfied, it suffices that the following orthogonality
conditions hold:
\[
\varphi_k = (\varphi, v_k) = 0, 
\qquad 
\Psi_k = (\Psi, v_k) = 0, 
\qquad 
k \in K_{0,\rho_1}.
\]
\end{remark}

\section{Uniqueness}\label{sec:6}

\begin{theorem}
Let the functions $\varphi(x)$ and $\Psi(x)$ be continuous in the closed domain $\Omega$, and let $K_{0,\rho_1}$ be empty.
Then the problem \eqref{1.1}--\eqref{1.4} admits at most one classical solution $(u(x,t),f(x))$.
\end{theorem}

\begin{proof}
Assume that, under the conditions of the theorem, there exist two pairs of solutions,
$(u_1(x,t),f_1(x))$ and $(u_2(x,t),f_2(x))$.
We prove that
\[
u(x,t)=u_1(x,t)-u_2(x,t)\equiv 0,
\qquad
f(x)=f_1(x)-f_2(x)\equiv 0.
\]

Since the problem under consideration is linear, the pair $(u(x,t),f(x))$ satisfies the following problem:
\begin{equation}\label{6.1}
\sum_{j=1}^M q_j \partial_t^{\rho_j} u(x,t) + A(x,D)u(x,t) = f(x)g(t),
\quad x \in \Omega,\; 0 < t \le T,
\end{equation}
\begin{equation}\label{6.2}
u(x,0) = 0,
\quad x \in \Omega,
\end{equation}
\begin{equation}\label{6.3}
B_j u(x,t) = \sum_{|\alpha|\le m_j} b_{\alpha,j}(x) D^\alpha u(x,t) = 0,
\end{equation}
\[
0 \le m_j \le m-1,\quad j=1,\dots,l,\quad x \in \partial\Omega,\; 0 < t \le T,
\]
\begin{equation}\label{6.4}
u(x,t_0)=0,
\quad x \in \Omega.
\end{equation}

Let $v_k(x)$ be an arbitrary eigenfunction of the spectral problem \eqref{2.1}--\eqref{2.2}
corresponding to the eigenvalue $\lambda_k$.
Consider the function
\begin{equation}\label{6.5}
w_k(t)=\int_{\Omega} u(x,t)v_k(x)\,dx.
\end{equation}

From \eqref{6.1}, we obtain
\[
\sum_{j=1}^M q_j \partial_t^{\rho_j} w_k(t)
=
\sum_{j=1}^M q_j \int_{\Omega} \partial_t^{\rho_j} u(x,t)v_k(x)\,dx
\]
\[
=
- \int_{\Omega} A(x,D)u(x,t)v_k(x)\,dx
+ g(t)\int_{\Omega} f(x)v_k(x)\,dx.
\]

Using the definition of eigenfunctions, we see that
\[
\sum_{j=1}^M q_j \partial_t^{\rho_j} w_k(t)
=
- \int_{\Omega} u(x,t)A(x,D)v_k(x)\,dx + f_k g(t)
=
- \lambda_k w_k(t) + f_k g(t).
\]

Consequently, $w_k(t)$ satisfies the Cauchy-type problem
\[
\sum_{j=1}^M q_j \partial_t^{\rho_j} w_k(t) + \lambda_k w_k(t)
=
f_k g(t),
\qquad
w_k(0)=0,
\]
\begin{equation}\label{6.6}
w_k(t_0)=0.
\end{equation}
By \eqref{5.5}, the solution of this problem has the form
\[
w_k(t)=f_k b_{k,\rho_1}(t).
\]
Using the additional condition \eqref{6.6}, we obtain
\begin{equation}\label{6.7}
w_k(t_0)=f_k b_{k,\rho_1}(t_0)=0.
\end{equation}
Since $K_{0,\rho_1}$ is empty, we have
\[
b_{k,\rho_1}(t_0)\neq 0
\quad \text{for all } k.
\]
Therefore, $f_k=0$ for all $k$.
Hence, $w_k(t)\equiv 0$ for all $k$.
By the completeness of the system of eigenfunctions $v_k$, it follows that
\[
f(x)\equiv 0,
\qquad
u(x,t)\equiv 0.
\]
This proves the uniqueness of the solution to the inverse problem \eqref{1.1}--\eqref{1.4}.
\end{proof}

\begin{remark}
If $K_{0,\rho_1}$ is not empty, the solution is not unique. Indeed, if $b_{k,\rho_1}(t_0)=0$ for some $k$, then relation \eqref{6.7} is satisfied for arbitrary values of $f_k$.
Hence, the inverse problem admits infinitely many solutions.
\end{remark}

\section{Special case of the operator $A(x,D)$}
\label{sec:7}
Paper \cite{27} studies general self-adjoint extensions in $L^2(\Omega)$ of elliptic operators
initially defined on $C_0^\infty(\Omega)$. The self-adjoint operator $\widehat{A}$ used in this
work belongs to this class. In particular, it follows from \cite{27} that functions from the
Sobolev space $L^{\tau m}_2(\Omega)$ with compact support in $\Omega$ are contained in the
domain of the  $\widehat{A}^\tau$.

As an application of Theorem \ref{theorem2} and Theorem \ref{theorem3} , we consider the inverse problem on the torus
$\mathbb{T}^N$,\quad $\mathbb{T}^N=(-\pi,\pi]^N$. Let
\[
A(D)=\sum_{|\alpha|=m} a_\alpha D^\alpha,
\]
be a homogeneous elliptic, symmetric, and positive operator with constant coefficients. We
study the multi-term time-fractional inverse problem
\begin{equation}\label{7.1}
\sum_{j=1}^M q_j  \partial_t^{\rho_j} u(x,t) + A(D)u(x,t) = f(x)g(t),
\quad (x,t) \in \mathbb{T}^N \times (0,T],
\end{equation}
\begin{equation}\label{7.2}
u(x,0) = \varphi(x), 
\quad x\in\mathbb{T}^N,
\end{equation}
\begin{equation}\label{7.3}
u(x,t_0) = \Psi(x), 
\quad x\in\mathbb{T}^N.
\end{equation}

All functions involved are assumed to be $2\pi$-periodic in each spatial variable. The closure
$\widehat{A}$ of $A(D)$ in $L^2(\mathbb{T}^N)$ is self-adjoint and admits the orthonormal
eigenfunctions
\[
(2\pi)^{-N/2} e^{i n\cdot x}, \quad n\in\mathbb{Z}^N,
\]
with eigenvalues $A(n)$. By the spectral theorem, for $\tau\ge0$, 
\[
\widehat{A}^\tau h(x)=\sum_{n\in\mathbb{Z}^N} A^\tau(n)\, h_n e^{i n\cdot x},
\]
and its domain is
\begin{equation}\label{7.4}
D(\widehat{A}^\tau)=
\left\{
h\in L^2(\mathbb{T}^N):
\sum_{n\in\mathbb{Z}^N} A^{2\tau}(n)\, |h_n|^2<\infty
\right\}.
\end{equation}
To define this domain in terms of Sobolev spaces, recall that a function $h\in L^2(\mathbb{T}^N)$ belongs to the Sobolev space $L^a_2(\mathbb{T}^N)$ \cite{28}, for a real $a>0$, if and only if 
\begin{equation}\label{7.5}
\|h\|_{L^a_2(\mathbb{T}^N)}^2
=
\sum_{n\in \mathbb{Z}^N} (1 + |n|^2)^{a} |h_n|^2 < \infty,
\end{equation}
where $h_n$ are the Fourier coefficients of $h$ with respect to the orthonormal eigenfunctions
$(2\pi)^{-N/2} e^{i n \cdot x}$ in $L^2(\mathbb{T}^N)$. 
It can readily be verified that there exist constants $c_1$ and $c_2$ such that one has the estimates
\[
c_1 \left( 1 + |n|^2 \right)^{\tau m}
\le
1 + A^{2 \tau} (n)
\le
c_2 \left( 1 + |n|^2 \right)^{\tau m}.
\]
Consequently, by comparing the expressions \eqref{7.4} and \eqref{7.5}, we conclude that 
$$
D(\widehat{A}^\tau) = L^{\tau m}_2(\mathbb{T}^N).
$$
\begin{theorem}\label{theorem5}
Let $\varphi, \Psi \in L_2^{\tau +m}(\mathbb{T}^N)$ where $\tau>\frac{N}{2}$ and $g\in C[0,T], \, g(t)\neq 0$.
Then there exists a classical solution $(u(x,t), f(x))$ of the inverse problem \eqref{7.1}--\eqref{7.3}, which can be represented in the following form      
$$
    f(x)=\sum_{n\in \mathbb{Z}^N} \frac{\left(\Psi_n-\varphi_n \left[ 1 - A(n) t_0^{\rho_1} E_{ \rho', \rho_1 + 1} \left( -A(n) t_0^{\rho_1}, * \right) \right]\right)e^{inx}  }{b_{n,\rho_1}(t_0)},
    $$
    $$
    u(x,t)= \sum_{n\in \mathbb{Z}^N}\varphi_n \left[ 1 - A(n) t^{\rho_1} E_{ \rho', \rho_1 + 1} \left( -A(n) t^{\rho_1}, * \right) \right]e^{inx}+\sum_{n\in\mathbb{Z}^N}f_
    n b_{n,\rho_1}(t) e^{inx}.
    $$
\end{theorem}
\begin{theorem}\label{theorem6}
    Let $\varphi, \Psi \in L_2^{\tau +m}(\mathbb{T}^N)$ where $\tau>\frac{N}{2}$ and $g\in C^1[0,T], \, g(0)\neq 0,$\\
    1) If the set $K_{0,\rho_1}$ is empty i.e. $b_{n,\rho_1}\neq 0$, for all $n$ then there exists a unique solution of the inverse problem \eqref{7.1}--\eqref{7.3}:
    $$
    f(x)=\sum_{n\in \mathbb{Z}^N} \frac{\left(\Psi_n-\varphi_n \left[ 1 - A(n) t_0^{\rho_1} E_{ \rho', \rho_1 + 1} \left( -A(n) t_0^{\rho_1}, * \right) \right]\right)e^{inx}  }{b_{n,\rho_1}(t_0)},
    $$
    $$
    u(x,t)= \sum_{n\in \mathbb{Z}^N}\varphi_n \left[ 1 - A(n) t^{\rho_1} E_{ \rho', \rho_1 + 1} \left( -A(n) t^{\rho_1}, * \right) \right]e^{inx}+\sum_{n\in\mathbb{Z}^N}f_
    n b_{n,\rho_1}(t) e^{inx}.
    $$
    2)If the set $K_{0,\rho_1}$ is not empty, for the existence of a solution to the inverse problem, it is necessary and sufficient that the following conditions
    $$
    \Psi_n=\varphi_n \left[ 1 - A(n) t_0^{\rho_1} E_{ \rho', \rho_1 + 1} \left( -A(n) t_0^{\rho_1}, * \right) \right],
\quad\quad n\in K_{0,\rho_1},    $$
    be satisfied. In this case, the solution to the problem \eqref{7.1}--\eqref{7.3}  exists but it is not unique: 
 $$
    f(x)=\sum_{n\in K_{\rho_1}} \frac{\left(\Psi_n-\varphi_n \left[ 1 - A(n) t_0^{\rho_1} E_{ \rho', \rho_1 + 1} \left( -A(n) t_0^{\rho_1}, * \right) \right]\right)e^{inx}  }{b_{n,\rho_1}(t_0)}+\sum_{n\in K_{0,\rho_1}}f_ne^{inx},
 $$
 $$
    u(x,t)= \sum_{n\in \mathbb{Z}^N}\left(\varphi_n \left[ 1 - A(n) t^{\rho_1} E_{ \rho', \rho_1 + 1} \left( -A(n) t^{\rho_1}, * \right) \right] +f_n b_{n,\rho_1}(t)\right)e^{inx},
    $$
    where $f_n, n\in K_{0,\rho_1}$ are arbitrary constants.
\end{theorem}
The proofs of the Theorems \ref{theorem5} and \ref{theorem6} are completely analogous to the proof of the Theorems \ref{theorem2} and \ref{theorem3}. However, it should be emphasized that in Lemma \ref{lem1} the operator $A$ is positive,
whereas the operator considered in problem \eqref{7.1} is non-negative.
Therefore, in the proofs of Theorems \ref{theorem5} and \ref{theorem6}, we employ the operator $A + I$,
where $I$ is the identity operator in $L_2(\Omega)$.

\section{Conclusion}\label{sec:8}

In this paper, we investigated an inverse source problem for a multi-term time-fractional diffusion equation with the Caputo derivatives. The source term was assumed in separable form $f(x)g(t)$, and the unknown spatial component was recovered from an additional over-determination condition imposed at an interior time.

The elliptic part of the equation is described by a general 
self-adjoint positive operator $A(x,D)$ of order $m \ge 2$.  By means of the spectral decomposition associated with the operator, the inverse problem was reduced to a sequence of multi-term fractional ordinary differential equations. Their solutions were expressed in terms of multinomial Mittag--Leffler functions.

A detailed asymptotic analysis of the corresponding multinomial Mittag--Leffler functions was carried out. Sharp asymptotic estimates were established and used to derive uniform lower bounds for the spectral denominator appearing in the reconstruction formula. These estimates play a crucial role in ensuring well-posedness of the inverse problem in the non-degenerate case.

Under appropriate smoothness assumptions on the given data, sufficient conditions guaranteeing the existence of a classical solution were obtained. The uniform and absolute convergence of the associated Fourier series was rigorously justified. Furthermore, uniqueness was proved under the condition that the set of vanishing spectral denominators is empty. In the degenerate case, necessary and sufficient compatibility conditions were derived, and the structure of the non-uniqueness was completely characterized.

The abstract results were further specified for homogeneous elliptic operators with constant coefficients on the torus, where the domain of fractional powers of the operator was explicitly identified with Sobolev spaces. This example illustrates the applicability and structural clarity of the developed spectral approach.

The obtained results provide a comprehensive analytical framework for multi-term time-fractional inverse source problems and extend existing well-posedness theory to a broad class of elliptic operators.

\section*{Acknowledgements}
This research no external funding.

\end{document}